\documentclass{amsart}

\usepackage[hmargin=3.5cm,top=3cm,bottom=3cm,includehead]{geometry}

\usepackage{amsmath, amssymb, amsthm, mathrsfs}
\usepackage[all]{xy}
\usepackage{microtype}
\usepackage{graphicx,color,enumerate}

\usepackage[pdfborder={0 0 0}]{hyperref}

\newcommand\N{{\mathbb N}}
\newcommand\R{{\mathbb R}}

\newcommand\C{{\mathbb C}}

\newcommand\Z{{\mathbb Z}}

\def\AA{{\mathcal A}}
\def\BB{{\mathcal B}}
\def\CC{{\mathcal C}}
\def\DD{{\mathcal D}}
\def\EE{{\mathcal E}}

\def\LL{{\mathcal L}}

\def\NN{{\mathcal N}}

\def\SS{{\mathcal S}}

\def\BBB{{\mathscr B}}
\def\CCC{{\mathscr C}}

\def\eps{{\varepsilon}}

\newcommand{\la}{\langle}
\newcommand{\ra}{\rangle}

\newtheorem{thm}{Theorem}[section]
\newtheorem{prop}[thm]{Proposition}
\newtheorem{lem}[thm]{Lemma}
\newtheorem{cor}[thm]{Corollary}
\theoremstyle{remark}
\newtheorem{rem}[thm]{Remark}

\theoremstyle{definition}
\newtheorem{definition}[thm]{Definition}

\numberwithin{equation}{section}
\newcommand{\beqn}{\begin{equation}}
\newcommand{\eeqn}{\end{equation}}
\newcommand{\bear}{\begin{eqnarray}}
\newcommand{\eear}{\end{eqnarray}}
\newcommand{\bean}{\begin{eqnarray*}}
\newcommand{\eean}{\end{eqnarray*}}
\newcommand{\bal}{\begin{aligned}}
\newcommand{\eal}{\end{aligned}}

\newcommand{\Black}{\color{black}}

\title[Exponential convergence for the Landau equation]{Exponential convergence to equilibrium for the homogeneous Landau equation with hard potentials}
\author{Kleber Carrapatoso}

\address{Ceremade, Universit\'e Paris Dauphine\\
Place du Mar\'echal De Lattre De Tassigny\\
75775 Paris cedex 16, France.}

\curraddr{CMLA, \'Ecole Normale Sup\'erieure de Cachan\\
61 av. du pr\'esident Wilson\\
94235 Cachan, France.}

\email{carrapatoso@cmla.ens-cachan.fr}

\subjclass[2000]{47H20, 76P05, 82B40, 35K55}
\keywords{Landau equation; spectral gap; exponential decay; hypodissipativity; hard potentials}

\begin{document}

\begin{abstract}
This paper deals with the long time behaviour of solutions to the spatially homogeneous Landau equation with hard potentials. We prove an exponential in time convergence towards the equilibrium with the optimal rate given by the spectral gap of the associated linearised operator. This result improves the polynomial in time convergence obtained by Desvillettes and Villani \cite{DesVi2}. Our approach is based on new decay estimates for the semigroup generated by the linearised Landau operator in weighted (polynomial or stretched exponential) $L^p$-spaces, using a method developed by Gualdani, Mischler and Mouhot \cite{GMM}.
\end{abstract}

\maketitle

\tableofcontents

\section{Introduction and main results}

This work deals with the asymptotic behaviour of solutions to the spatially homogeneous Landau equation for hard potentials. It is well known that these solutions converge towards the Maxwellian equilibrium when time goes to infinity and we are interested in quantitative rates of convergence.

\smallskip

On the one hand, in the case of Maxwellian molecules, Villani \cite{Vi1} and Desvillettes-Villani \cite{DesVi2} have proved a linear functional inequality between the entropy and entropy dissipation by constructive methods, from which one deduces an exponential convergence (with quantitative rate) of the solution to the Landau equation towards the Maxwellian equilibrium in relative entropy, which in turn implies an exponential convergence in $L^1$-distance (thanks to the Csisz\'ar-Kullback-Pinsker inequality).
This kind of linear functional inequality relating entropy and entropy dissipation is known as Cercignani's Conjecture in Boltzmann and Landau theory, for more details and a review of results we refer to \cite{DMV}.

On the other hand, in the case of hard potentials, Desvillettes-Villani \cite{DesVi2} proves a functional inequality for entropy-entropy dissipation that is not linear, from which one obtains a polynomial convergence of solutions towards the equilibrium, again in relative entropy, which implies the same type of convergence in $L^1$-distance.

\smallskip

Before going further on details of existing results and on the contributions of the present work, we shall introduce in a precise manner the problem addressed here.
In kinetic theory, the Landau equation is a model in plasma physics that describes the evolution of the density in the phase space of all positions and velocities of particles. Assuming that the density function does not depend on the position, we obtain the \emph{spatially homogeneous Landau equation} in the form
\beqn\label{eq:landau}
\left\{
\bal
\partial_t f  &=  Q(f,f)  \\
f_{|t=0} &= f_0 ,
\eal
\right.
\eeqn
where $f=f(t,v) \geq 0$ is the density of particles with velocity $v$ at time $t$, $v\in\R^3$ and $t\in\R^+$. The Landau operator $Q$ is a bilinear operator given by
\beqn\label{eq:oplandau0}
Q(g,f) = \partial_{i} \int_{\R^3} a_{ij}(v-v_*) \left[ g_* \partial_j f - f \partial_{j}g_*\right]\, dv_*,
\eeqn
where here and below we shall use the convention of implicit summation over repeated indices and we use the shorthand $g_* = g(v_*)$, $\partial_{j} g_* = \partial_{v_{*j}} g(v_*)$, $f=f(v)$ and $\partial_j f = \partial_{v_j} f(v)$. 

The matrix $a$ is nonnegative, symmetric and depends on the interaction between particles. If two particles interact with a potential proportional to $1/r^s$, where $r$ denotes their distance, $a$ is given by (see for instance \cite{Villani-BoltzmannBook})
\beqn\label{eq:aij}
a_{ij}(v) = |v|^{\gamma+2}\left( \delta_{ij} - \frac{v_i v_j}{|v|^2}\right),
\eeqn
with $\gamma=(s-4)/s$. We usually call hard potentials if $\gamma\in(0,1]$, Maxwellian molecules if $\gamma=0$, soft potentials if $\gamma \in (-3,0)$ and Coulombian potential if $\gamma=-3$. Through this paper we shall consider the case of hard potentials $\gamma\in (0,1]$.

\smallskip

The Landau equation conserves mass, momentum and energy. Indeed, at least formally, for any test function $\varphi$ we have (see e.g. \cite{Vi2})
$$
\int_{\R^3} Q(f,f) \varphi(v) \, dv = \frac12 \int_{\R^3 \times \R^3} a_{ij}(v-v_*) f f_* 
\left(\frac{\partial_i f}{f} -  \frac{\partial_{i} f_*}{f_*}  \right)\left( \partial_j \varphi - \partial_j \varphi_* \right) \, dv \, dv_*
$$
from which we deduce
\beqn\label{eq:cons}
\int Q(f,f) \varphi(v) = 0 \qquad\text{for}\qquad \varphi(v) = 1, v, |v|^2.
\eeqn
Moreover, the entropy $H(f)=\int f\log f$ is nonincreasing. Indeed, at least formally, since $a_{ij}$ is nonnegative, we have the following inequality for the entropy dissipation $D(f)$,
\begin{equation}\label{eq:Df}
\begin{aligned}
D(f) :&= -\frac{d}{dt} H(f) \\
&=
\frac{1}{2}\int_{\R^3\times \R^3} f f_* \,a_{ij}(v-v_*)
\left( \frac{\partial_i f}{f} -  \frac{\partial_{i*} f_*}{f_*} \right)
 \left( \frac{\partial_j f}{f} -  \frac{\partial_{j*} f_*}{f_*} \right) \,dv \,dv_*
\geq 0.
\end{aligned}
\end{equation}
It follows that any equilibrium is a Maxwellian distribution
$$
\mu_{\rho,u,T} (v) := \frac{\rho}{(2\pi T)^{3/2}} e^{-\frac{|v-u|^2}{2T}},
$$
for some $\rho >0$, $u\in \R^3$ and $T>0$. This is the Landau version of the famous Boltzmann's $H$-theorem (for more details we refer to \cite{DesVi2,Vi1} again), from which the solution $f(t,\cdot)$ of the Landau equation is expected to converge towards the Maxwellian $\mu_{\rho_f, u_f, T_f}$ when $t\to +\infty$, where $\rho_f$ is the density of the gas, $u_f$ the mean velocity and $T_f$ the temperature, defined by
$$
\rho_f = \int f(v), \quad
u_f = \frac{1}{\rho} \int v f(v),\quad
T_f = \frac{1}{3\rho}\int |v-u|^2 f(v),
$$
and these quantities are defined by the initial datum $f_0$ thanks to the conservation properties of the Landau operator \eqref{eq:cons}. 

We may only consider the case of initial datum $f_0$ satisfying
\beqn\label{f0}
\int_{\R^3} f_0(v)\, dv = 1, \quad 
\int_{\R^3} v f_0(v)\, dv = 0, \quad
\int_{\R^3} |v|^2 f_0(v) \, dv = 3,
\eeqn
the general case being reduced to \eqref{f0} by a simple change of coordinates (see \cite{DesVi2}).
Then, we shall denote $\mu(v) = (2\pi)^{-3/2} e^{-|v|^2/2}$ the standard Gaussian distribution in $\R^3$, which corresponds to the Maxwellian with $\rho=1$, $u=0$ and $T=1$, i.e. the Maxwellian with same mass, momentum and energy of $f_0$ \eqref{f0}.

\smallskip

We linearise the Landau equation around $\mu$, with the perturbation 
$$
f=\mu + h,
$$
hence the equation satisfied by $h=h(t,v)$ takes the form
\beqn\label{eq:lin}
\partial_t h = \LL h + Q(h,h),
\eeqn
with initial datum $h_0$ defined by $h_0 =f_0 - \mu$, and 
where the linearised Landau operator $\LL$ is given by
\beqn\label{eq:oplandaulin}
\LL h= Q(\mu,h)+Q(h,\mu).
\eeqn
Furthermore, from the conservations properties \eqref{eq:cons}, we observe that the null space of $\LL$ has dimension $5$ and is given by (see e.g. \cite{DL,Guo,BM,M,MS})
\beqn\label{eq:kerLL}
\NN(\LL)  = \mathrm{Span} \{\mu, v_1\mu,v_2 \mu, v_3\mu, |v|^2 \mu\}.
\eeqn

\medskip

\subsection{Known results}\label{ssec:known}
We present here existing results concerning spectral gap estimates for the linearised operator and convergence to equilibrium for the nonlinear equation.

For any weight function $m=m(v)$ ($m : \R^3 \to \R^+$) we define the weighted Lebesgue space $L^p(m)$, for $p\in[1,+\infty]$, associated to the norm
$$
\| f\|_{L^p(m)} := \| m f \|_{L^p},
$$
and the weighted Sobolev spaces $W^{s,p}(m)$ for $s\in \N$, associated to the norm
$$
\bal
\| f\|_{W^{s,p}(m)} &:= \left( \sum_{|\alpha|\leq s}\| \partial^\alpha f \|_{L^p(m)}^p \right)^{1/p}, \quad\text{if } p\in [1,+\infty), \\
\| f \|_{W^{s,\infty}(m)} &:= \sup_{|\alpha| \leq s} \| \partial^\alpha f \|_{L^\infty(m)}    .
\eal
$$
We denote by $\DD$ the Dirichlet form associated to $-\LL$ on $L^2(\mu^{-1/2})$,
$$
\DD(h) := \la -\LL h, h \ra_{L^2(\mu^{-1/2})} := \int (-\LL h) h \mu^{-1},
$$
and we say that $h \in \NN(\LL)^{\perp}$, where $\NN(\LL)$ denotes the nullspace of $\LL$, if $h$ is of the form $h = h - \Pi h$, where $\Pi$ denotes the projection onto the null space.
It is easy to observe that $\LL$ is self-adjoint on $L^2(\mu^{-1/2})$ and $\DD(h) \ge 0$, which implies that the spectrum of $\LL$ on $L^2(\mu^{-1/2})$ is included in $\R^-$.

We can now state the existing results on the spectral gap of $\LL$ on $L^2(\mu^{-1/2})$.
The spectral gap inequality for the linearised Landau operator for hard potentials $\gamma \in (0,1]$,
\beqn\label{eq:GapDLandauBM}
\DD(h)  
\geq \lambda_0 \, \|  h \|_{L^2(\mu^{-1/2})}^2, \quad \forall\, h \in \NN(\LL)^{\perp},
\eeqn 
was proven by Baranger-Mouhot \cite{BM}, for some constructive constant $\lambda_0 >0$.

In the case of hard and soft potentials $\gamma \in (-3,1]$,  Mouhot \cite{M} proved the following result
\beqn\label{eq:GapDLandauM}
\DD(h) 
\geq \lambda_0 \left\{  \| h  \|_{H^1(\la v \ra^{\gamma/2} \mu^{-1/2})}^2 + \| h \|_{L^2(\la v \ra^{(\gamma+2)/2} \mu^{-1/2})}^2 \right\},  \quad 
\forall\,  h \in \NN(\LL)^\perp.
\eeqn

Furthermore, Guo \cite{Guo}, by nonconstructive arguments, and later Mouhot-Strain \cite{MS}, by constructive arguments, proved a spectral gap inequality for an anisotropic norm for the linearised Landau operator (in all cases: hard, soft and Coulombian potentials) $\gamma \in [-3,1]$,
\beqn\label{eq:GapDLandauMS}
\bal
 \DD(h) 
\geq \lambda_0 \| h \|_*^2,
\quad \forall\,  h \in \NN(\LL)^\perp,
\eal
\eeqn
with the anisotropic norm $\| \cdot \|_{*}$ defined by
$$
\| h \|_*^2 := \|\la v \ra^{\gamma/2} P_v \nabla h  \|_{L^2( \mu^{-1/2})}^2 
  +\|\la v \ra^{(\gamma+2)/2} (I-P_v) \nabla h  \|_{L^2(\mu^{-1/2})}^2 
  +\|\la v \ra^{(\gamma+2)/2}  h  \|_{L^2( \mu^{-1/2})}^2  
$$
where $P_v$ denotes the projection onto the $v$-direction, more precisely $P_v g = \left( \frac{v}{|v|}\cdot g \right) \frac{v}{|v|}$. We also have from \cite{Guo}, the reverse inequality
\beqn\label{eq:GapDLandauMS2}
\DD(h) \leq C_2 \| h \|_*^2, 
\quad \forall\,  h \in \NN(\LL)^\perp,
\eeqn
which, together with \eqref{eq:GapDLandauMS}, imply a spectral gap for $\LL$ in $L^2(\mu^{-1/2})$ if and only if $\gamma+2\geq 0$.

Summarising the results \eqref{eq:GapDLandauBM}, \eqref{eq:GapDLandauM} and \eqref{eq:GapDLandauMS}, in the case of hard potentials $\gamma \in (0,1]$ and Maxwellian molecules $\gamma=0$, there is a constructive constant $\lambda_0 >0$ (spectral gap) such that
\beqn\label{eq:lambda0bis}
 \DD(h) 
\geq \lambda_0 \| h \|_{L^2(\mu^{-1/2})}^2,
\quad \forall\,  h \in \NN(\LL)^\perp.
\eeqn
As a consequence, considering the linearised Landau equation $\partial_t h = \LL h$, we have an exponential decay 
\beqn\label{eq:lambda0}
\forall\, t\geq 0,\; \forall\, h\in L^2(\mu^{-1/2}), \qquad
\| \SS_{\LL}(t) h - \Pi h \|_{L^2(\mu^{-1/2})} \leq  e^{-\lambda_0 t} \| h - \Pi h\|_{L^2(\mu^{-1/2})},
\eeqn
where $\SS_{\LL}(t)$ denotes the semigroup generated by $\LL$ and $\Pi$ the projection onto $\NN(\LL)$, the null space of $\LL$ given by \eqref{eq:kerLL}.

\medskip

Another approach is to study directly the nonlinear equation, establishing functional inequalities between the entropy and the entropy dissipation.
The following entropy dissipation inequality for the (nonlinear) Landau operator for Maxwellian molecules $\gamma=0$  
\beqn\label{eq:dissipLandauDV}
D (f) \geq \delta_0 \, H(f| \mu ),   \quad \forall f \in L^1_{1,0,1}(\R^3) :=  \left\{f \in L^1(\R^3) ; \rho_f = 1, u_f = 0, T_f=1 \right\},
\eeqn
for some explicit constant $\delta_0$, was proven by Desvillettes-Villani \cite{DesVi2} and Villani \cite{Vi1}. Here $H (f | \mu) := \int f \log (f/\mu)$ denotes the relative entropy of $f$ with respect to $\mu$, and this inequality implies an exponential decay to the equilibrium $\mu$.
Taking $f = \mu + \eps h$, they also deduce a degenerated spectral gap inequality for the linearised Landau operator for $\gamma=0$,
\beqn\label{eq:GapDLandauH1}
\DD(h) \geq \bar\delta_0 \, \|\nabla h \|_{L^2(\mu^{-1/2})}^2 \quad \forall\,  h \in \NN(\LL)^\perp.
\eeqn

In the case of hard potentials $\gamma\in(0,1]$, Desvillettes-Villani \cite{DesVi2} proved the following entropy-entropy dissipation inequality, for some explicit $\delta_1, \delta_2 >0$,
\beqn\label{eq:dissipLandauDV2}
D (f) \geq \min\left\{ \delta_1 H(f|\mu), \delta_2 H(f|\mu)^{1+\gamma/2}    \right\}  \quad \forall f \in L^1_{1,0,1}(\R^3),
\eeqn
which implies a polynomial decay to equilibrium in relative entropy (see Theorem~\ref{thm:DV} for more details).

\medskip

As we can see above, the result \eqref{eq:dissipLandauDV2} tell us that the solution to the Landau equation converges to the equilibrium in polynomial time. Furthermore, from the exponential decay for the linearised equation \eqref{eq:lambda0bis}-\eqref{eq:lambda0}, we might expect that the solution to the nonlinear equation also decays exponentially in time if it lies in some neighbourhood of the equilibrium in which the linear part is dominant.
One could then expect  to prove an exponential convergence to equilibrium combining these to results: for small times one uses the polynomial decay, then for large times, when the solution enters in the appropriated neighbourhood of the equilibrium (in $L^2(\mu^{-1/2})$-norm), one uses the exponential decay. However these two theories, linear and nonlinear, are not compatible in the sense that the spectral gap for the linearised operator holds in $L^2(\mu^{-1/2})$ and the Cauchy theory \cite{DesVi1} for the nonlinear Landau equation is constructed in $L^1$-spaces with polynomial weight, which means that in order to apply the strategy above, starting from some initial datum in weighted $L^1$-space, one would need the appearance of the $L^2(\mu^{-1/2})$-norm of the solution in positive time to be able to use \eqref{eq:lambda0bis}-\eqref{eq:lambda0}, and this is not known to be true (one does not know even if the $L^2(\mu^{-1/2})$-norm is propagated).
Hence, in order to be able to "connect" the linearised theory with the nonlinear one, we need to enlarge the functional space of semigroup decay estimates generated by the linearised operator $\LL$.

\medskip

Our goal in this paper is to prove an \emph{(optimal) exponential in time convergence} of solutions to the Landau equation towards the equilibrium and our strategy is based on:

\begin{enumerate}[(1)]

\item New decay estimates for the semigroup generated by the linearised Landau operator $\LL$ in various $L^p$-spaces with polynomial and stretched exponential weight, using a method developed in \cite{GMM}.

\item The well-known Cauchy theory for the nonlinear equation developed in \cite{DesVi1,DesVi2}: the appearance and uniform propagation of $L^1$-polynomial moments, smoothing effect and the polynomial in time convergence to equilibrium.

\item The strategy of connecting the linearised theory with the nonlinear one, roughly presented in the above paragraph.

\end{enumerate}

\subsection{Statement of the main result}
Let us state our main result, which proves a sharp exponential decay to equilibrium for the spatially homogeneous Landau equation with hard potentials.

First of all we define the notion of weak solutions that we shall use.

\begin{definition}[Weak solutions \cite{DesVi1}]
Let $\gamma \in (0,1]$ and consider a nonnegative initial data with finite mass, momentum and energy $f_0 \in L^1 (\la v \ra^2) $. 
We say that $f$ is a weak solution of the Cauchy problem \eqref{eq:landau} if the following conditions are fulfilled:

\begin{enumerate}[($i$)]

\item $f\ge 0$, $f \in C ([0,\infty); \DD') \cap L^\infty([0,\infty); L^1(\la v \ra^2) ) \cap L^1_{loc}([0,\infty); L^1 (\la v \ra^{2+\gamma}))$;

\item for any $t\ge 0$
$$
\int f(t) |v|^2 \le \int f_0 |v|^2 
$$

\item $f$ verifies \eqref{eq:landau} in the distributional sense: for any $\varphi \in C^\infty_c([0,\infty) \times \R^3_v)$, for any $t \ge 0$,
$$
\int f(t) \varphi(t) - \int f_0 \varphi(0) - \int_0^t \int f(\tau) \partial_t \varphi(\tau)
= \int_0^t \int Q(f(\tau), f(\tau)) \varphi(\tau),
$$ 
where the last integral in the right-hand side is defined by 
$$
\int Q(f,f) \varphi = \frac12 \iint a_{ij}(v-v_*) (\partial_{ij} \varphi + \partial_{ij}\varphi_*) \, f_* f 
+ \iint b_{i}(v-v_*) (\partial_{i} \varphi - \partial_{i}\varphi_*) \, f_* f 
$$

\end{enumerate}

\end{definition}

It is proven in \cite{DesVi1} that if $f_0 \in L^1( \la v \ra^{2+\delta})$ for some $\delta >0$, then there exists a global weak solution.

\medskip

Our main theorem reads:

\begin{thm}[Exponential decay to equilibrium]\label{thm:rate}
Let $\gamma\in (0,1]$ and a nonnegative $f_0\in L^1(\la v \ra^{2+\delta})$ for some $\delta>0$, satisfying \eqref{f0}. Then, for any weak solution $(f_t)_{t\geq 0}$ to the spatially homogeneous Landau equation \eqref{eq:landau} with initial datum $f_0$, there exists a constant $C>0$ such that
$$
\forall\, t\geq 0, \qquad
\| f_t - \mu \|_{L^1} \leq C e^{-\lambda_0 t},
$$
where $\lambda_0>0$ is the spectral gap \eqref{eq:lambda0bis}-\eqref{eq:lambda0} of the linearised operator $\LL$ on $L^2(\mu^{-1/2})$.

\end{thm}

As mentioned above,
in the case of hard potentials $\gamma\in (0,1]$, a polynomial decay to equilibrium was proven by Desvillettes and Villani \cite{DesVi2} and in the case of Maxwellian molecules $\gamma=0$ an exponential decay to equilibrium was proven by Villani~\cite{Vi1} and also by Desvillettes and Villani~\cite{DesVi2}. The proof of Theorem~\ref{thm:rate} relies on coupling the polynomial in time decay from \cite{DesVi2} for small times and the exponential decay for the linearised operator in weighted $L^p$-spaces from Theorem~\ref{thm:trou} for large times, when the linearised dynamics is dominant. This method was first used by Mouhot \cite{Mouhot2} where is proved the exponential decay to equilibrium for the spatially homogeneous Boltzmann equation for hard potentials with cut-off. Later, the same approach was used by Gualdani, Mischler and Mouhot \cite{GMM} to prove the exponential decay to the equilibrium for the inhomogeneous Boltzmann equation for hard spheres on the torus, and also by Mischler and Mouhot \cite{MiMo-FP} for Fokker-Planck equations.

\subsection{Organisation of the paper}
We start Section~\ref{sec:linear} presenting some properties of the linearised equation and then we state and prove the "spectral gap/semigroup decay" extension theorem (Theorem~\ref{thm:trou}), which is a key ingredient of the proof of the main theorem. 
Finally, in Section~\ref{sec:nonlinear}, we prove estimates for the (nonlinear) Landau operator and then prove Theorem~\ref{thm:rate}.

\bigskip
\noindent{\bf Acknowledgements.}
We would like to thank St\'ephane Mischler and Cl\'ement Mouhot for enlightened discussions and their encouragement.

\section{The linearised equation}\label{sec:linear}

We define (see e.g. \cite{DesVi1,Vi2,Vi1}) in $3$-dimension the following quantities
\beqn\label{eq:bc}
b_i(z) = \partial_j a_{ij}(z) = - 2 \, |z|^\gamma \, z_i, \quad c(z) =  \partial_{ij} a_{ij}(z)  = - 2 (\gamma+3) \, |z|^\gamma.
\eeqn
Hence, we can rewrite the Landau operator \eqref{eq:oplandau0} in the following way
\beqn\label{eq:oplandau}
Q(g,f) = ( a_{ij}*g) \partial_{ij} f - (c*g) f = \partial_i [(a_{ij}*g)\partial_j f - (b_i*g)f].
\eeqn
We also denote
\beqn\label{eq:barabc}
\bar a_{ij}(v) = a_{ij}*\mu, \quad
\bar b_i(v) = b_i * \mu, \quad
\bar c(v) = c*\mu.
\eeqn

Using the form \eqref{eq:oplandau} of the operator $Q$, we decompose the linearised Landau operator $\LL$ defined 
in \eqref{eq:oplandaulin} as $\LL =  \AA_0 + \BB_0$, where we define
\beqn\label{eq:A0B0}
\bal
\AA_0 f &:= Q(f,\mu) = (a_{ij}* f)\partial_{ij}\mu - (c * f)\mu, \\
\BB_0 f &:= Q(\mu,f) =(a_{ij}* \mu)\partial_{ij}f - (c * \mu)f.
\eal
\eeqn
Consider a smooth nonnegative function $\chi \in C^\infty_c(\R^3)$ such that $0\leq \chi(v) \leq 1$, $\chi(v) \equiv 1$ for $|v|\leq 1$ and $\chi(v) \equiv 0$ for $|v|>2$. For any $R\geq 1$ we define $\chi_R(v) := \chi(R^{-1}v)$ and in the sequel we shall consider the function $M\chi_R$, for some constant $M>0$.
Then, we make the final decomposition of the operator $\LL$ as $\LL = \AA + \BB$ with
\beqn\label{eq:AB}
\AA  := \AA_0 + M \chi_R ,
\qquad
\BB  := \BB_0 - M\chi_R,
\eeqn
where $M$ and $R$ will be chosen later (see Lemma~\ref{lem:hypo}).

\medskip
Let us now make our assumptions on the weight functions $m=m(v)$.
We define the polynomial weight, for all $p\in[1,+\infty)$,
\beqn\label{m1}
\bal
m = \la v \ra^k, \quad\text{with } k >  \gamma+2 +3(1-1/p)
\eal
\eeqn
and the abscissa
\beqn\label{amp1}
\bal
a_{m,p} &:= 2[3(1-1/p) - k],& \text{ if } \gamma=0, \\
a_{m,p} &:= - \infty ,& \text{ if } \gamma \in (0,1]. 
\eal
\eeqn
Moreover, we define the exponential weight, for $p\in[1,+\infty)$,
\beqn\label{m2}
\bal
m = \exp\left(r \la v\ra^s \right), \quad\text{with}\quad
\left\{
\begin{array}{ll}
r>0, &\text{if } s\in(0,2), \\
0< r < \frac{1}{2p}, & \text{if } s=2,
\end{array}
\right.
\eal
\eeqn
and we define the abscissa, for all cases,
\beqn\label{amp2}
\bal
a_{m,p} &:= - \infty.
\eal
\eeqn

We are able know to state the following result on the exponential decay of the semigroup associated to the Landau linearised operator $\LL$ in various weighted $L^p$-spaces. Observe that this result extends the functional space in which a semigroup decay estimate is already known to hold, as presented in \eqref{eq:lambda0bis}-\eqref{eq:lambda0} for the space $L^2(\mu^{-1/2})$. We include here the case of Maxwellian molecules $\gamma=0$ for the sake of completeness.

\begin{thm}\label{thm:trou}
Let $\gamma\in [0,1]$, $p\in[1,2]$, a weight function $m=m(v)$ satisfying \eqref{m1} or \eqref{m2} and their respective abscissa $a_{m,p}$ given by \eqref{amp1} or \eqref{amp2}. Consider the linearised Landau operator $\LL$ \eqref{eq:oplandaulin}, 
then for any positive $\lambda~\leq~\min\{\lambda_0, \lambda_1 \}$, for any $\lambda_1 < |a_{m,p}|$, there exists $C_\lambda>0$ such that
\beqn\label{eq:thmtrou}
\forall\, t\geq 0,\; \forall\, h\in L^p(m), \qquad
\| \SS_\LL (t) h - \Pi h \|_{L^p(m)} \leq C_\lambda \, e^{-\lambda t} \,  \| h - \Pi h \|_{L^p(m)},
\eeqn
where $\SS_\LL (t) h$ is the semigroup generated by $\LL$, $\Pi $ is the projection onto the null space of $\LL$, and 
$\lambda_0>0$ is the spectral gap of $\LL$ in $L^2(\mu^{-1/2})$ given by \eqref{eq:lambda0bis}-\eqref{eq:lambda0}.
\end{thm}

\begin{rem}\label{rem:trou}
As we can see in the definition of $a_{m,p}$ in \eqref{amp1} and \eqref{amp2}, we conclude that:
\begin{enumerate}

\medskip

\item \textit{Hard potentials case $\gamma\in (0,1]$}: for both weight functions $m$, stretched exponential weight \eqref{m2} or polynomial weight \eqref{m1}, we have an exponential in time decay with optimal rate $\lambda = \lambda_0$, since $a_{m,p}:=-\infty$. 

\medskip

\item \textit{Maxwellian molecules case $\gamma = 0$}: if $m$ is a stretched exponential weight \eqref{m2}, we get the optimal rate $\lambda = \lambda_0$, since $a_{m,p}:=-\infty$; if $m$ is a polynomial weight \eqref{m1}, then we get the optimal rate $\lambda = \lambda_0$ if $k$ is big enough such that $a_{m,p} = 2[3(1-1/p) - k] < -\lambda_0$, otherwise we have $\lambda < 2[k - 3(1-1/p)] $.

\end{enumerate}
 
\end{rem}

This theorem extends the exponential semigroup decay to weighted $L^p$ spaces using a method developed by Gualdani, Mischler and Mouhot \cite{GMM} (see Theorem~\ref{thm:extension} below) for Boltzmann and Fokker-Planck equations (see also Mischler and Mouhot \cite{MiMo-FP} for other results on Fokker-Planck equations).

\subsection{Abstract theorem}
We shall present in this subsection an abstract theorem from \cite{GMM,MiMo-FP}, which will be used to prove Theorem \ref{thm:trou}.

Let us introduce some notation before state the theorem. 
Consider two Banach spaces $(X,\| \cdot \|_X )$ and $(Y, \|\cdot\|_Y)$. We denote by $\BBB(X,Y)$ the space of bounded linear operators from $X$ to $Y$ and by $\| \cdot \|_{\BBB(X,Y)}$ its operator norm. Moreover we write $\CCC(X,Y)$ the space of closed unbounded linear operators from $X$ to $Y$ with dense domain. When $X=Y$ we simply denote $\BBB(X) = \BBB(X,X)$ and $\CCC(X) = \CCC(X,X)$. 

Given a Banach space $X$ and a operator $\Lambda: X \to X$, we denote $\SS_\Lambda(t)$ or $e^{t\Lambda}$ the semigroup generated by $\Lambda$. We also denote $\NN(\Lambda)$ its null space, $\mathrm{dom}(\Lambda)$ its domain, $\Sigma(\Lambda)$ its spectrum and $\mathrm R(\Lambda)$ its range. Recall that for any $z$ in the resolvent set $\rho(\Lambda) := \C \setminus \Sigma(\Lambda)$, the operator $\Lambda - z$ is invertible, moreover the resolvent operator $(\Lambda-z)^{-1} \in \BBB(X)$ and its range equals $\mathrm{dom}(\Lambda)$. An eigenvalue $\xi \in \Sigma(\Lambda)$ is isolated if
$$
\Sigma(\Lambda) \cap \{ z \in \C; \; |z-\xi| \leq r \} = \{ \xi \} 
\quad\text{for some } r>0.
$$
Then for an isolated eigenvalue $\xi$ we define the associated spectral projector $\Pi_{\Lambda,\xi} \in \BBB(X)$ by
\beqn\label{eq:spectral-projector}
\Pi_{\Lambda,\xi} := -\frac{1}{2 i \pi} \int_{|z-\xi|=r'} (\Lambda-z)^{-1} \, dz
\quad\text{with } 0 < r' < r.
\eeqn
If moreover the algebraic eigenspace $\mathrm R (\Pi_{\LL,\xi})$ is finite dimensional, we say that $\xi$ is a discrete eigenvalue and write $\xi \in \Sigma_d(\Lambda)$. Finally, for any $a\in \R$ we define the subspace 
$$
\Delta_a := \{ z \in \mathbb C ; \Re  z >a \}.
$$

\begin{definition}
Let $X_1$, $X_2$ and $X_3$ be Banach spaces and $\SS_1 \in L^1(\R_+, \BBB(X_1,X_2))$, 
$\SS_2 \in L^1(\R_+, \BBB(X_2,X_3))$. We define the convolution $\SS_2 * \SS_1 \in L^1(\R_+, \BBB(X_1,X_3))$ by 
$$
\forall\, t\geq 0, \qquad
\SS_2 * \SS_1 (t) := \int_{0}^t \SS_2(s) \SS_1(t-s)\, ds.
$$
If $X_1=X_2=X_3$ and $\SS=\SS_1=\SS_2$, we define $\SS^1 = \SS$ and $\SS^{*n} = \SS*\SS^{*(n-1)}$ for all $n\geq 2$.

\end{definition}

We can now state a simplified version of \cite[Theorem 2.13]{GMM} that is suitable for our particular case.

\begin{thm}
\label{thm:extension}
Let $E$ and $\EE$ be Banach spaces such that $E \subset \EE$ is dense with continuous embedding. Consider the operators $L \in \CCC(E)$, $\LL \in \CCC(\EE)$ with $L = \LL_{|E}$ and assume that:

\begin{enumerate}[(1)]

\item L generates a semigroup $\SS_L(t)$ on $E$, $L$ is hypo-dissipative on $\mathrm R(I-\Pi_{\LL,0})$ and moreover 

\begin{enumerate}[(i)]


\item There exists $\lambda_0 >0$ such that 
$$
\Sigma(L) \cap \Delta_{b}   =\{ 0 \}, \quad \text{for any}\quad -\lambda_0 < b < 0.
$$

\item There is $b' < - \lambda_0$ such that
$$
\Sigma(L) \cap \Delta_{b'}   =\{ 0, - \lambda_0 \}.
$$

\end{enumerate}

\item There are $\AA,\BB \in \CCC(\EE)$ such that $\LL = \AA + \BB$, with the corresponding restrictions $A = \AA_{|E}$ and $B=\BB_{|E}$ on $E$, some $n \in \N^*$, some $a \in \R$ and some constant $C_a >0$ such that

\begin{enumerate}[(i)]
\item $\BB-a$ is hypo-dissipative on $\EE$;

\item $A\in \BBB(E)$ and $\AA \in \BBB(\EE)$;

\item we have
$$
\left\| (\AA\SS_\BB)^{*n}(t) \right\|_{\BBB(\EE,E)} \leq C_a e^{at}.
$$
\end{enumerate}

\end{enumerate}

Then $\LL$ is hypo-dissipative on $\EE$ and we have the following estimates: If 
$a< - \lambda_0$, there holds
\beqn\label{eq:thm-LL-dissipative}
\forall\, t\geq0, \qquad
\left\| \SS_\LL(t) - \Pi_{\LL,0} \right\|_{\BBB(\EE)} 
\leq C' \,  e^{- \lambda_0 t}.
\eeqn
Otherwise, if $a \ge \lambda_0$, then for any $a' > a$ there holds
\beqn\label{eq:thm-LL-dissipative-bis}
\forall\, t\geq0, \qquad
\left\| \SS_\LL(t) - \Pi_{\LL,0} \right\|_{\BBB(\EE)} 
\leq C' \,  e^{a' t},
\eeqn
where $C' >0$ is an explicit constant depending on the constants from the assumptions. 

\end{thm}

\Black

This theorem permits us to \emph{enlarge the space of spectral/semigroup estimates} of a given operator. More precisely,
the knowledge of the spectral information in some ``small space"~(1) allows us to extend this information to a ``bigger space" (\eqref{eq:thm-LL-dissipative} or \eqref{eq:thm-LL-dissipative-bis}), when the operator satisfies some conditions (2).

In our case, the spectral gap estimate of $\LL$ on $L^2(\mu^{-1/2})$ stated in \eqref{eq:lambda0bis}-\eqref{eq:lambda0} gives assumption~(1) of Theorem~\ref{thm:extension}. Thus, in order to prove Theorem~\ref{thm:trou}, we consider the operators $\AA$ and $\BB$ defined in \eqref{eq:AB}, and we shall prove assumptions (2i), (2ii) and (2iii) on the space $\EE = L^p(m)$. We can then conclude to the semigroup decay estimates \eqref{eq:thm-LL-dissipative} or \eqref{eq:thm-LL-dissipative-bis} applying Theorem~\ref{thm:extension}, which is nothing but the estimate in Theorem~\ref{thm:trou}.

\subsection{Hypo-dissipativity properties}
\label{ssec:hypo}
In this subsection we shall investigate the hypo-dissipativity of the operator $\BB$, defined in \eqref{eq:AB}, on $L^p(m)$ spaces, in order to prove assumption {\it (2i)} of Theorem~\ref{thm:extension}.
Before proving the desired result in Lemma~\ref{lem:hypo}, we give the following lemmas that will be useful in the sequel.

\begin{lem}\label{lem:Jalpha}
Let $J_\alpha(v) := \int_{\R^3} |v-w|^\alpha \mu(w)\, dw$, for $0 \leq \alpha \leq 3$, and denote $M_\alpha (\mu) := \int |v|^\alpha \mu$. Then it holds:
\begin{enumerate}[(a)]

\item $J_0(v) = 1$.

\item $J_\alpha(v) \leq |v|^\alpha + M_\alpha(\mu)$, for $ 0 <\alpha \leq 1$. 

\item $J_\alpha(v)  \leq |v|^\alpha + M_2(\mu)^{\alpha/2}$, for $1< \alpha < 2$.

\item $J_2(v) = |v|^2 + M_2(\mu)$.

\item $J_\alpha(v)  \leq |v|^\alpha +10^{\alpha/4}|v|^{\alpha/2}+ M_4(\mu)^{\alpha/4}$, 
for $2< \alpha \leq 3$.

\end{enumerate}

\begin{rem}
As we will see in the proof of Lemma~\ref{lem:hypo}, 
the important point here is that, for all $0\leq \alpha \leq 3$, the dominant part of the upper bound of $J_\alpha$ has coefficient equals to $1$.
\end{rem}

\end{lem}

\begin{proof}[Proof of Lemma~\ref{lem:Jalpha}]
Items (a) and (d) are evident. For (b) we see that $|v-w|^\alpha \leq |v|^\alpha + |w|^\alpha$ and it implies $J_\alpha(v) \leq |v|^\alpha + M_\alpha(\mu)$. To prove item (c) we use $\alpha/2 < 1$ and Jensen's inequality to write
$$
J_\alpha(v)  \leq \left(\int_{\R^3} |v-w|^2 \mu(dw)\right)^{\alpha/2} = \left(|v|^2 + M_2(\mu)\right)^{\alpha/2} \leq |v|^\alpha + M_2(\mu)^{\alpha/2}.
$$
Finally, item (e) can be proven in the same way as (d). Firstly, for $\alpha=4$ explicit computation gives $J_4(v) = |v|^4 + 10|v|^2 + M_4(\mu)$. Then, from $\alpha/4 <1$ and Jensen's inequality we obtain
$$
\bal
J_\alpha(v) \leq \left(\int_{\R^3} |v-w|^4 \mu(dw)\right)^{\alpha/4} 
&= \left(|v|^4 + 10|v|^2 + M_4(\mu)\right)^{\alpha/4} \\
&\leq |v|^\alpha +10^{\alpha/4}|v|^{\alpha/2}+ M_4(\mu)^{\alpha/4}.
\eal
$$
\end{proof}

Furthermore we have the following results concerning $\bar a_{ij}(v)$. 
\begin{lem}\label{lem:bar-aij}
The following properties hold:

\begin{enumerate}[(a)]

\item The matrix $\bar a(v)$ has a simple eigenvalue $\ell_1(v)>0$ associated with the eigenvector $v$ and a double eigenvalue $\ell_2(v)>0$ associated with the eigenspace $v^{\perp}$. Moreover,
$$
\bal
\ell_1(v) &= \int_{\R^3} \left(1 - \left(\frac{v}{|v|}\cdot\frac{w}{|w|}   \right)^2   \right) |w|^{\gamma+2} \mu(v-w)\, dw\\
\ell_2(v) &= \int_{\R^3} \left(1 - \frac12 \left| \frac{v}{|v|}\times\frac{w}{|w|}  \right|^2   \right) |w|^{\gamma+2} \mu(v-w)\, dw .
\eal
$$
When $|v|\to +\infty$ we have
$$
\bal
\ell_1(v) &\sim  2 |v|^\gamma \\
\ell_2(v) &\sim  |v|^{\gamma+2} .
\eal
$$
If $\gamma\in (0,1]$ there exists $\ell_0 >0$ such that, for all $v\in\R^3$, 
$\min\{ \ell_1(v), \ell_2(v) \} \geq \ell_0$.

\item The function $\bar a_{ij}$ is smooth, for any multi-index $\beta\in \N^3$
$$
|\partial^\beta \bar a_{ij}(v)| 
\leq C_\beta \la v \ra^{\gamma+2-|\beta|}
$$
and
$$
\bal
\bar a_{ij}(v) \xi_i \xi_j &= \ell_1(v) |P_v \xi|^2 + \ell_2(v)|(I-P_v)\xi|^2, \\
\bar a_{ij}(v) v_i v_j  &= \ell_1(v) |v|^2 ,
\eal
$$
where $P_v$ is the projection on $v$, i.e.
$$
P_v \xi = \left( \xi \cdot \frac{v}{|v|} \right) \frac{v}{|v|}.
$$

\item We have
$$
\bar a_{ii}(v) 
= 2 \int_{\R^3} |v-v_*|^{\gamma+2} \mu(v_*)\, dv_*
\qquad\text{and}\qquad
\bar b_i(v) = - \ell_1(v)\, v_i.
$$

\end{enumerate}

\end{lem}

\begin{proof}[Proof of Lemma~\ref{lem:bar-aij}]
We just give the proof of item $(c)$ since $(a)$ comes from \cite[Propositions 2.3 and 2.4, Corollary 2.5]{DL} and $(b)$ is \cite[Lemma 3]{Guo}.

Hence, for item $(c)$ we write
$$
\bar a_{ii}(v) = \sum_{i=1}^3 \int_{\R^3} a_{ii}(v-v_*) \mu(v_*) \, dv_*.
$$
Using \eqref{eq:aij} we obtain that 
$$
a_{ii}(z) = \sum_{i=1}^3 |z|^{\gamma+2}\left( 1 - \frac{z_i^2}{|z|^2} \right)
= 2 |z|^{\gamma+2}
$$
and then
$$
\bar a_{ii}(v) = 2 \int_{\R^3} |v-v_*|^{\gamma+2} \mu(v_*)\, dv_*.
$$
Moreover, we compute
$$
\bar b_i(v) = (\partial_{j} a_{ij} * \mu)(v) = (a_{ij} * \partial_j \mu)(v)
= - \int_{\R^3} a_{ij}(v-v_*) v_{*j} \,  \mu(v_*)\, dv_*,
$$
and using that $a_{ij}(z) z_j = 0 $ we obtain
$$
\bal
\bar b_i(v) &= - \int_{\R^3} a_{ij}(v-v_*) v_{*j} \, \mu(v_*)\, dv_*\\
&=- \int_{\R^3} a_{ij}(v_*) (v_j-v_{*j}) \, \mu(v-v_{*})\, dv_* \\
&= - \left(\int_{\R^3} a_{ij}(v_*) \mu(v-v_*)\, dv_*\right) v_j
= - \bar a_{ij}(v) v_j = - \ell_1(v)\, v_i.
\eal
$$
\end{proof}

With the help of the results above, we are able to state the hypo-dissipativity result for $\BB$.
\begin{lem}\label{lem:hypo}
Let $\gamma\in[0,1]$, $p\in[1,+\infty)$ and consider a weight function $m=m(v)$ satisfying \eqref{m1} or \eqref{m2} with the corresponding the abscissa \eqref{amp1} or \eqref{amp2}, respectively. Then, for any $a> a_{m,p}$ we can choose $M$ and $R$ large enough such that the operator $\BB-a$ is dissipative in $L^p(m)$, in the sense that
$$
\forall\, t\ge 0, \qquad \| \SS_\BB(t) \|_{\BBB(L^p(m))} \le e^{ at }.
$$
\end{lem}

\begin{proof}[Proof of Lemma~\ref{lem:hypo}]
We split the proof into four steps.

\medskip
\noindent
{\it Step 1.}
Let us denote $\Phi'(z) = |z|^{p-1}\mathrm{sign}(z)$ and consider the equation
$$
\partial_t f = \BB f = \BB_0 f -M\chi_R f.
$$
For all $1\leq p <+\infty$, we have
\beqn\label{eq:SSf}
\bal
\frac{d}{dt} \| f \|_{L^p(m)} 
&=\| f\|_{L^p(m)}^{1-p} \left\{ \int (\BB f) \Phi'(f) m^p \right\} \\
&= \| f\|_{L^p(m)}^{1-p} \left\{ \int (\BB_0 f) \Phi'(f) m^p - \int (M\chi_R f) \Phi'(f) m^p  \right\}
\eal
\eeqn
with, from \eqref{eq:A0B0} and \eqref{eq:oplandau},
\begin{align*}
\int (\BB_0 f) \Phi'(f) \, m^p
&= \int \bar a_{ij} \partial_{ij} f  \, \Phi'(f) \, m^p
- \int \bar c  \, m^p \, |f|^p 
\end{align*}
Let us denote $h=m^{\theta} f$, for some $\theta$ to be chosen later. For the first term, using $\Phi'(f) = \Phi'(h)\, m^{-\theta(p-1)}$, we have
\begin{align*}
T_1 &= \int \bar a_{ij} \partial_{ij} (h m^{-\theta})  \, \Phi'(h) \, m^{p+\theta(1-p)} \\
&=
- \int \partial_j (h m^{-\theta}) \partial_i \left( \bar a_{ij} \Phi'(h) m^{p+\theta(1-p)}  \right) \\
&= - \int \partial_j (h m^{-\theta})\bar a_{ij} \partial_i \left(  \Phi'(h) m^{p+\theta(1-p)}   \right)
- \int \partial_j (h m^{-\theta}) \bar b_j \, \Phi'(h) m^{p+\theta(1-p)} \\
&=: T_{11} + T_{12}.   
\end{align*}
We also have 
$$
\bal
&\partial_j (h m^{-\theta})\partial_i \left(  \Phi'(h) m^{p+\theta(1-p)}   \right) \\
&\qquad= (p-1)\partial_i h \partial_j h \, m^{p(1-\theta) \,}|h|^{p-2}
+ \frac{[p + \theta(1-p)]}{p} \partial_i m  \partial_j (|h|^p) \, m^{p(1-\theta)-1} \\
&\qquad- \frac{\theta(p-1)}{p} \partial_i (|h|^p) \partial_j m \, m^{p(1-\theta)-1}
-\theta[p-\theta(p-1)] \partial_i m \partial_j m \, m^{p(1-\theta)-2} \, |h|^p,
\eal
$$
then, since $\bar a_{ij}$ is symmetric , it follows
\begin{align*}
T_{11} &=  - (p-1 )\int \bar a_{ij} \partial_i h \partial_j h \, m^{p(1-\theta)} \,|h|^{p-2}\\
&+\left[ 2\theta \frac{(p-1)}{p} - 1 \right] \int \bar a_{ij} \partial_i m \partial_j (|h|^p) \, m^{p(1-\theta)-1} \\
&+ \theta[p-\theta(p-1)] \int \bar a_{ij} \partial_i m \partial_j m \, m^{p(1-\theta)-2}\, |h|^{p}.
\end{align*}
Performing an integration by parts, we obtain
\beqn\label{eq:T11}
\bal
T_{11} &=  - (p-1 )\int \bar a_{ij} \partial_i h \partial_j h \, m^{p(1-\theta)} \,|h|^{p-2}\\
&\quad+\delta_1(p,\theta) \int  \bar b_{i} \partial_i m  \, m^{p(1-\theta)-1}  \,|h|^p \\
&\quad+\delta_1(p,\theta) \int \bar a_{ij} \partial_{ij} m  \, m^{p(1-\theta)-1}  \, |h|^p\\
&\quad+\delta_2(p,\theta) 
\int \bar a_{ij} \partial_i m  \partial_j  m \,  m^{p(1-\theta)-2}  \, |h|^p 
\eal
\eeqn
where
\beqn\label{eq:delta}
\delta_1(p,\theta) :=   1 - 2\theta (1-1/p), 
\qquad
\delta_2(p,\theta) := \delta_1(p,\theta)[p(1-\theta)-1]
+\theta[p-\theta(p-1)].
\eeqn
For the term $T_{12}$ we have
\beqn\label{eq:T12}
\bal
T_{12} &=   
- \int \partial_j (h m^{-\theta}) \bar b_j \, \Phi'(h) m^{p+\theta(1-p)} \\
&= - \int \partial_j h \Phi'(h)  \bar b_j \,  m^{p(1-\theta)} 
 + \theta  \int  h \Phi'(h)  \bar b_j \partial_j m \,  m^{p(1-\theta)-1} \\
&= - \frac1p \int \partial_j ( |h|^p )  \bar b_j \,  m^{p(1-\theta)} 
+ \theta  \int   \bar b_j \partial_j m \,  m^{p(1-\theta)-1}\, |h|^p \\
&=  \frac1p \int  \bar c \,  m^{p(1-\theta)}  \, |h|^p
+   \int   \bar b_j \partial_j m \,  m^{p(1-\theta)-1}\, |h|^p .
\eal
\eeqn
Gathering \eqref{eq:T11} and \eqref{eq:T12} one obtains
\beqn\label{eq:BB0}
\int (\BB_0 f) \Phi'(f) m^p = -(p-1) \int \bar a_{ij} \partial_i (m^\theta f) \partial_j(m^\theta f)  \, m^{p-2\theta}\, |f|^{p-2} + \int \varphi_{m,p,\theta}(v)\, m^p \, |f|^p,
\eeqn
with
\beqn\label{phi_mp}
\bal
\varphi_{m,p,\theta} 
&:=  \delta_1(p,\theta)\left( \bar a_{ij} \,\frac{\partial_{ij} m}{m}\right)
+\delta_2(p,\theta) \left(\bar a_{ij} \, \frac{\partial_i m }{m} \,\frac{\partial_j m}{m} \right) \\
&\quad 
+ (1+\delta_1(p,\theta))\left(\bar b_i \, \frac{\partial_i m}{m}\right)
+ \left(\frac1p-1\right) \bar c,
\eal
\eeqn
where $\delta_1$ and $\delta_2$ are defined in \eqref{eq:delta}.

Let us now split the proof into two different cases: polynomial weight $m$ 
satisfying \eqref{m1} and stretched exponential weight $m$ verifying \eqref{m2}.

\medskip
\noindent
{\it Step 2. Polynomial weight.}
Consider $m= \la v \ra^k$ defined in \eqref{m1}. On the one hand, we have
$$
\bal
\frac{\partial_i m}{m} = k v_i \la v \ra^{-2}, \qquad
\frac{\partial_i m}{m}\, \frac{\partial_j m}{m} = k^2 v_i v_j \la v \ra^{-4},\\
\frac{\partial_{ij} m}{m} = \delta_{ij} \,  k \la v \ra^{-2} + k(k-2) v_i v_j \la v \ra^{-4}.
\eal
$$
Hence, from the definitions \eqref{eq:bc}-\eqref{eq:barabc} and Lemma~\ref{lem:bar-aij} we obtain
\beqn\label{eq:mpoly1}
\bal
\bar a_{ij} \,\frac{\partial_{ij} m}{m} &= (\delta_{ij}\bar a_{ij}) \,  k \la v \ra^{-2} +  (\bar a_{ij} v_i v_j)  \, k(k-2)\la v \ra^{-4} \\
&= \bar a_{ii} \,  k \la v \ra^{-2} +  \ell_1(v)  \, k(k-2) |v|^2 \la v \ra^{-4},
\eal
\eeqn
where we recall that the eigenvalue $\ell_1(v) >0$ is defined in Lemma~\ref{lem:bar-aij}. 
Moreover, arguing exactly as above we obtain
\beqn\label{eq:mpoly2}
\bal
\bar a_{ij} \,\frac{\partial_i m}{m}\, \frac{\partial_j m}{m} = (\bar a_{ij} v_i v_j)  \, k^2\la v \ra^{-4}
=   \ell_1(v)  \, k^2 |v|^2 \la v \ra^{-4}
\eal
\eeqn
and also, using the fact that $\bar b_{i}(v) = - \ell_1(v) v_i$ from Lemma~\ref{lem:bar-aij}, 
\beqn\label{eq:mpoly3}
\bal
\bar b_{i} \,\frac{\partial_i m}{m} =  - \ell_1(v) v_i  \, k v_i  \la v \ra^{-2}
=   - \ell_1(v)  \, k |v|^2 \la v \ra^{-2}.
\eal
\eeqn

On the other hand, from item (c) of Lemma~\ref{lem:bar-aij} and definitions \eqref{eq:bc}-\eqref{eq:barabc} we obtain that 
\beqn\label{eq:barac}
\bar a_{ii} = 2 J_{\gamma+2}(v) 
\quad\text{and}\quad 
\bar c = -2(\gamma+3) J_{\gamma}(v),
\eeqn
where $J_\alpha$ is defined in Lemma~\ref{lem:Jalpha}.
It follows from \eqref{phi_mp}--\eqref{eq:barac} that 
\beqn\label{eq:phi1}
\bal
\varphi_{m,p,\theta}(v)
&= 
\delta_1(p,\theta) \,  2 k J_{\gamma+2}(v) \la v \ra^{-2}
+ \delta_1(p,\theta)\, k(k-2) \, \ell_1(v) \, |v|^2 \la v \ra^{-4} \\
&\quad
+\delta_2(p,\theta)\,  k^2  \, \ell_1(v) \, |v|^2 \la v \ra^{-4} 
- [1+\delta_1(p,\theta)]\,   k \, \ell_1(v) \, |v|^2 \la v \ra^{-2}\\
&\quad
+2(\gamma+3)\left(1-\frac1p \right)  J_{\gamma}(v).
\eal
\eeqn 
Since $\ell_1(v) \sim 2 \la v \ra^\gamma$ and $J_\alpha(v) \sim \la v\ra^\alpha$ when $|v|\to +\infty$ by Lemmas \ref{lem:bar-aij} and \ref{lem:Jalpha}, the dominant terms in \eqref{eq:phi1} are the first, fourth and fifth one, all of order $\la v \ra^{\gamma}$.

For $p\in (1,+\infty)$ we choose $\theta = p/[2(p-1)]$, then $\delta_1(p,\theta) = 0$, 
$\delta_2(p,\theta) =  p^2/[4(p-1)]$ and
$$
\bal
\varphi_{m,p,\theta}(v) &= 
\frac{p^2}{4(p-1)}  \,  k^2  \, \ell_1(v) \, |v|^2 \la v \ra^{-4}  -   k\, \ell_1(v) \, |v|^2 \la v \ra^{-2}
+2(\gamma+3)\left(1-\frac1p \right)  J_{\gamma}(v).
\eal
$$
Using Lemma~\ref{lem:Jalpha} to bound $J_\gamma$, we obtain that
\beqn\label{cas:poly}
\left\{
\bal
&\limsup_{|v| \to \infty} \varphi_{m,p,\theta}(v) \le -2 \left[k-3(1-1/p)\right] , & \text{if } \gamma=0 ,\\
&\limsup_{|v| \to \infty} \varphi_{m,p,\theta}(v) \le -2 \left[k-(\gamma+3)(1-1/p)\right]\la v\ra^\gamma,   & \text{if } \gamma \in (0,1],
\eal
\right.
\eeqn
and we recall that $k > (\gamma + 3) (1-1/p) $ from \eqref{m1}.

If $p=1$, for all $\theta$, we have $\delta_1(1,\theta) = 1$ and $\delta_2(1,\theta)=0$ which gives
$$
\bal
\varphi_{m,1,\theta}(v) &= 2 k J_{\gamma+2}(v)  \la v \ra^{-2} + k(k-2) \lambda(v) |v|^2 \la v \ra^{-4} - 2k \ell_1(v) |v|^2 \la v \ra^{-2}, 
\eal
$$
and the dominant terms are the first and last one, both of order $\la v \ra^{\gamma}$. Using Lemma~\ref{lem:Jalpha} to bound $J_{\gamma+2}$, we obtain
\beqn\label{cas:poly2}
\left\{
\bal
&\limsup_{|v| \to \infty} \varphi_{m,1,\theta}(v)  \le   -2 k,  & \text{if } \gamma=0, \\
&\limsup_{|v| \to \infty} \varphi_{m,1,\theta}(v) \le -2k \la v\ra^\gamma,  & \text{if } \gamma \in (0,1].
\eal
\right.
\eeqn

\medskip
\noindent
{\it Step 3. Exponential weight.}
We consider now $m=\exp (r\la v \ra^s)$ given by \eqref{m2}. In this case we have
$$
\bal
\frac{\partial_i m}{m} = r s v_i \la v\ra^{s-2},
\qquad
\frac{\partial_i m}{m}\frac{\partial_j m}{m} = r^2 s^2 v_i v_j \la v\ra^{2s-4}, \\
\frac{\partial_{ij} m}{m} = r s \la v\ra^{s-2}  \delta_{ij}
+ r s (s-2) v_i v_j \la v\ra^{s-4}  + r^2 s^2 v_i v_j \la v\ra^{2s-4} .
\eal
$$
It follows from last equation that
\beqn\label{eq:mexp1}
\bal
\bar a_{ij} \,\frac{\partial_{ij} m}{m} &= (\delta_{ij}\bar a_{ij}) \,  r s \la v\ra^{s-2} +  (\bar a_{ij} v_i v_j)  \, r s (s-2)\la v\ra^{s-4} + (\bar a_{ij} v_i v_j)\, r^2 s^2 \la v\ra^{2s-4} \\
&= \bar a_{ii} \,  r s \la v\ra^{s-2} 
+  \ell_1(v)   \, r s (s-2) |v|^2 \la v\ra^{s-4} 
+ \ell_1(v) \, r^2 s^2 |v|^2 \la v\ra^{2s-4},
\eal
\eeqn
where we used Lemma~\ref{lem:bar-aij},
\beqn\label{eq:mexp2}
\bal
\bar a_{ij} \,\frac{\partial_i m}{m}\, \frac{\partial_j m}{m} = (\bar a_{ij} v_i v_j)  \, r^2 s^2 \la v\ra^{2s-4}
=   \ell_1(v)  \, r^2 s^2 |v|^2 \la v\ra^{2s-4}
\eal
\eeqn
and also, using the fact that $\bar b_{i}(v)= - \ell_1(v) v_i$ , 
\beqn\label{eq:mexp3}
\bal
\bar b_{i} \,\frac{\partial_i m}{m} =  - \ell_1(v) v_i  \, r s v_i \la v\ra^{s-2}
=   - \ell_1(v)  \, r s |v|^2 \la v\ra^{s-2}.
\eal
\eeqn

Gathering together \eqref{phi_mp}, \eqref{eq:mexp1}, \eqref{eq:mexp2} and \eqref{eq:mexp3}, and thanks to Lemma~\ref{lem:bar-aij}, it yields
\beqn\label{eq:phiexp}
\bal
\varphi_{m,p,\theta}(v) &= 
\delta_1(p,\theta) \, 2 r s  J_{\gamma+2}(v)  \la v\ra^{s-2}
+\delta_1(p,\theta) \, r s(s-2)  \ell_1(v) |v|^2    \la v\ra^{s-4}\\
&\quad
+\delta_1(p,\theta) \, r^2 s^2  \ell_1(v)  |v|^2  \la v\ra^{2s-4}
+\delta_2(p,\theta) \, r^2 s^2  \ell_1(v) |v|^2 \la v\ra^{2s-4} \\ 
&\quad
- [1+\delta_1(p,\theta)]  \, r s  \ell_1(v) |v|^2  \la v\ra^{s-2}
+2(\gamma+3)\left(1-\frac1p  \right)  \,  J_{\gamma}(v)
\eal
\eeqn 
where we recall that $J_\alpha$ is given in Lemma~\ref{lem:Jalpha}.

Let us choose $\theta=0$ for all cases $p\in[1,+\infty)$. Then $\delta_1(p,0) = 1$, 
$\delta_2(p,0) =  p-1$ and
\beqn\label{eq:phiexp1}
\bal
\varphi_{m,p,0}(v) &= 
 2 r s  J_{\gamma+2}(v)  \la v\ra^{s-2}
+  r s(s-2)  \ell_1(v) |v|^2    \la v\ra^{s-4}
+  p r^2 s^2  \ell_1(v)  |v|^2  \la v\ra^{2s-4}\\
&\quad 
 - 2 r s  \ell_1(v) |v|^2  \la v\ra^{s-2}
+2(\gamma+3)\left(1-\frac1p  \right)  \,  J_{\gamma}(v),
\eal
\eeqn
and we recall that $\ell_1(v) \sim 2 \la v \ra^\gamma$ and $J_\alpha(v) \sim \la v\ra^\alpha$ when $|v|\to +\infty$ by Lemmas \ref{lem:bar-aij} and \ref{lem:Jalpha}.

If $0< s <2$, the dominant terms in \eqref{eq:phiexp1} is the fourth one, of order $\la v\ra^{\gamma+s}$. Then we obtain the asymptotic behaviour
\beqn\label{cas:exp1}
\limsup_{|v| \to \infty} \varphi_{m,p,0}(v) \le -4 r s \la v\ra^{s+\gamma}
\eeqn
and we recall that $s+\gamma>0$.
If $s=2$, the dominant terms in \eqref{eq:phiexp1} are the first, third and fourth one, all of order $\la v\ra^{\gamma+2}$. Hence, using Lemma~\ref{lem:Jalpha} to bound $J_{\gamma+2}$ and Lemma~\ref{lem:bar-aij}, we obtain
\beqn\label{cas:exp2}
\limsup_{|v| \to \infty} \varphi_{m,p,0}(v) \le 4r\left(2p r - 1 \right) \la v\ra^{\gamma+2},
\eeqn
and we recall that $r < 1/(2p)  $ from \eqref{m2}.

\medskip
\noindent
{\it Step 4.}
Finally, gathering Steps 1, 2 and 3, for any $p\in[1,+\infty)$, for any $a>a_{m,p}$, thanks to the asymptotic behaviour of $\varphi_{m,p,\theta}$ in \eqref{cas:poly}-\eqref{cas:poly2}-\eqref{cas:exp1}-\eqref{cas:exp2}, we can choose $M$ and $R$ large enough such that $\varphi_{m,p,\theta}(v) - M\chi_R(v) \leq a$ for all $v\in \R^3$. It follows that the operator $\BB - a =\BB_0-M\chi_R - a$ is dissipative in $L^p(m)$, more precisely, for all 
$f\in L^p(m)$ we have  
\beqn\label{eq:SS_BB}
\forall \, t \ge 0, \quad
\| \SS_\BB(t) f\|_{L^p(m)} \leq e^{at} \| f \|_{L^p(m)}.
\eeqn
Indeed, from \eqref{eq:SSf} and \eqref{eq:BB0} we obtain
$$
\bal
\frac1p\frac{d}{dt} \| f \|_{L^p(m)}^p 
&\leq  -(p-1) \int \bar a_{ij} \partial_i(m^\theta f) \partial_j(m^\theta f) m^{p-2\theta} |f|^{p-2} + \int (\varphi_{m,p,\theta} - M\chi_R ) m^p |f|^p \\
&\leq \int (\varphi_{m,p,\theta} - M\chi_R ) m^p |f|^p \\
&\leq a \int m^p |f|^p
\eal
$$
which yields \eqref{eq:SS_BB}.
\end{proof}

\begin{rem}
Coming back to the case of exponential moment in Step 3, we could also, for $p\in(1,+\infty)$, chose $\theta = p/[2(p-1)]$ as we did for the polynomial weight. This would not change anything for $0<s<2$, however for the case $s=2$ we would obtain
$$
\limsup_{|v| \to \infty} \varphi_{m,p,\theta}(v) \le  \left(\frac{2p^2r^2}{p-1}   -4r\right) \la v\ra^{\gamma+2}
$$
which goes to $-\infty$ when $|v|\to +\infty$ if $r < 2(p-1)/p^2  $, modifying then the conditions on $r$ defined in \eqref{m2}. Using these two computations, a more general condition on $r$ defined in \eqref{m2} in the case $s=2$ would be $r < \max\left\{\frac{1}{2p}, \frac{2(p-1)}{p^2} \right\}$.
\end{rem}

\subsection{Regularisation properties}\label{ssec:reg}
We are now interested in regularisation properties of the operator $\AA$ and the iterated convolutions of $\AA\SS_\BB$, in order to prove assumptions \textit{(2ii)} and \textit{(2iii)} of Theorem~\ref{thm:extension}.
Let us recall the operator $\AA$ defined in \eqref{eq:AB},
$$
\AA g = \AA_0 g + M\chi_R g = (a_{ij}* g)\partial_{ij}\mu - (c* g)\mu + M\chi_R g,
$$
for $M$ and $R$ large enough chosen before.

Thanks to the function $\chi_R$, for any $q\in[1,+\infty)$, $p\geq q$ and any weight function $m_0$, we have
\beqn\label{eq:chiR}
\| M \chi_R g \|_{L^q(m_0)} \leq C \| \chi_R m_0 m^{-1} \|_{L^{pq/(p-q)}} \| g\|_{L^{p}(m)} 
\leq C \| g \|_{L^p(m)},
\eeqn
from which we deduce that $M \chi_R \in \BBB(L^p(m), L^q(m_0))$.

Let us now focus on regularisation estimates for the operator $\AA_0$. First of all we give the following result, which will be useful in the sequel.

\begin{lem}\label{lem:a*g}
Let $\gamma\in [0,1]$ and $\beta \in \N^3$ be a multi-index such that $|\beta|\leq 2$. Then
$$
|\partial_\beta(a_{ij}* g)(v)| \lesssim \la v \ra^{\gamma+2} \| \partial_\beta g \|_{L^1(\la v\ra^{\gamma+2})}
\quad\text{and}\quad
|\partial_\beta(a_{ij}* g)(v)| \lesssim \la v \ra^{\gamma+2-|\beta|} \| g \|_{L^1(\la v\ra^{\gamma+2-|\beta|})}
$$

\end{lem}

\begin{proof}[Proof of Lemma \ref{lem:a*g}]
First of all, we write $\partial_\beta(a_{ij}* g) = a_{ij}* \partial_\beta g$ and then
$$
|(a_{ij}* \partial_\beta g)(v)| \leq \int |a_{ij}(v-v_*)| |\partial_\beta g_*|\, dv_*.
$$
For $\gamma\in [0,1]$ we have $|a_{ij}(v-v_*)| \leq |v-v_*|^{\gamma + 2} \leq C \la v\ra^{\gamma+2} \la v_* \ra^{\gamma+2}$, which yields
$$
|(a_{ij}* \partial_\beta g)(v)| \lesssim \la v \ra^{\gamma+2} \| \partial_\beta g \|_{L^1(\la v\ra^{\gamma+2})}.
$$
Finally, writing $\partial_\beta(a_{ij}* g) = \partial_\beta a_{ij}*  g$ 
and using that 
$$
|\partial_\beta a_{ij} (v-v_*)| 
\lesssim |v-v_*|^{\gamma+2-|\beta|} 
\lesssim \la v\ra^{\gamma +2-|\beta|} \la v_*\ra^{\gamma+2-|\beta|}
$$
from Lemma~\ref{lem:bar-aij} and because $\gamma+2-|\beta| \geq 0$, it follows
$$
|(\partial_\beta a_{ij}* g)(v)| 
\lesssim \int \la v\ra^{\gamma +2-|\beta|} \la v_*\ra^{\gamma+2-|\beta|} |g_*| \, dv_*
\lesssim \la v \ra^{\gamma+2-|\beta|} \| g \|_{L^1(\la v\ra^{\gamma+2-|\beta|})},
$$
which finishes the proof.
\end{proof}

\begin{lem}\label{lem:A0}
Let $\gamma\in[0,1]$ and $p\in [1,+\infty]$. Then we have
\beqn\label{eq:lemregA0}
\|\AA_0 g \|_{L^p(m)}
\leq C_\mu \left( 
\| g\|_{L^{1}(\la v\ra^{\gamma+2})}
+ \| g\|_{L^{1}(\la v\ra^\gamma)} 
\right).
\eeqn
As a consequence, $\AA_0 \in \BBB(L^p(m), L^1(\la v \ra^{\gamma+2}))$ and also $\AA_0 \in \BBB(L^p(m))$.

\end{lem}

\begin{proof}[Proof of Lemma~\ref{lem:A0}]
For the first inequality, we write
$$
\|\AA_0 g \|_{L^p(m)} \leq \|(a_{ij}* g)\partial_{ij}\mu \|_{L^p(m)} 
+ \|(c* g)\mu \|_{L^p(m)} .
$$
For the first term, using Lemma~\ref{lem:a*g}, we compute
$$
\bal
\|(a_{ij}* g)\partial_{ij}\mu \|_{L^p(m)}^p 
&\leq C\,  \|  g \|_{L^1(\la v\ra^{\gamma+2})}^p \int \la v \ra^{(\gamma+2)p} |\partial_{ij} \mu(v)|^p m^p(v) \, dv \\
&\leq C_\mu \, \|  g \|_{L^1(\la v\ra^{\gamma+2})}^p .
\eal
$$
Arguing in the same way, we also obtain
$$
\bal
\|(c* g)\mu \|_{L^p(m)}^p 
&\leq C\,  \|  g \|_{L^1(\la v\ra^{\gamma})}^p \int \la v \ra^{\gamma p} |\mu(v)|^p m^p(v)\, dv \\
&\leq C_\mu \, \|  g \|_{L^1(\la v\ra^{\gamma})}^p ,
\eal
$$
which completes the proof of the first inequality of the lemma.

Then we compute, for some $\sigma>0$ and using H\"older's inequality,
$$
\bal
\| g\|_{L^{1}(\la v\ra^{\gamma+2})}
&\leq \left( \int  \la v \ra^{-\sigma p/(p-1)} \right)^{(p-1)/p}
\| g \|_{L^p(\la v \ra^{\gamma+2+\sigma})} \\
&\leq C \| g \|_{L^p(\la v \ra^{\gamma+2+\sigma})},
\eal
$$ 
if $\sigma > 3(1-1/p)$. This implies that $\|\AA_0 g \|_{L^p(m)} \leq C_\mu \| g \|_{L^p(m)}$ 
since $k > \gamma+2 + 3(1-1/p)$ when $m=\la v \ra^k$ satisfies \eqref{m1} or $m=e^{r\la v \ra ^s}$ satisfies $\eqref{m2}$.
\end{proof}

\begin{cor}\label{cor:AB}
Let $p \in [2,+\infty]$. Then $\AA \in \BBB (L^p(m), L^2(\mu^{-1/2}))$ and for any $a>a_{m,p}$ we have
$$
\| \AA\SS_\BB(t) \|_{\BBB (L^p(m), L^2(\mu^{-1/2}))} \leq C_a e^{at}.
$$

\end{cor}

\begin{proof}[Proof of Corollary~\ref{cor:AB}]
From Lemma~\ref{lem:A0} and equation \eqref{eq:chiR} it follows that $\AA \in \BBB (L^p(m), L^2(\mu^{-1/2}))$ for all $p \in [2,+\infty]$. Then we compute using Lemma~\ref{lem:hypo},
$$
\| \AA\SS_\BB(t) f \|_{L^2(\mu^{-1/2})} \leq \| \AA \|_{\BBB (L^p(m), L^2(\mu^{-1/2}))} \, 
\| \SS_\BB(t) f \|_{L^p(m)} \leq C e^{at} \| f\|_{L^p(m)},
$$
which concludes the proof.
\end{proof}

Let us denote $m_0 = e^{r \la v \ra ^2}$ with $r\in(0,1/4)$, then $L^2(\mu^{-1/2}) \subset L^q(m_0)$ for any $1\leq q \leq 2$.

\begin{lem}\label{lem:reg}
There exists  $C>0$ such that for all $1\leq p  < 2$,
\beqn\label{eq:regBB}
\| \SS_\BB(t) f \|_{L^2(m_0)} \leq C\, t^{-\frac32\left(\frac1p - \frac12 \right)} \, e^{a t}\, \|f \|_{L^p(m_0)}, \qquad \forall\, t\geq 0.
\eeqn
As a consequence, for all $1\leq p < 2$ and $m$ satisfying \eqref{m1} or \eqref{m2}, 
for any $a' > a$ we have
\beqn\label{eq:regAB}
\| (\AA\SS_\BB)^{*2}(t) f \|_{L^2(\mu^{-1/2})} \leq C\, e^{a' t}\, \|f \|_{L^p(m)}, \qquad \forall\, t\geq 0.
\eeqn

\end{lem}

\begin{proof}[Proof of Lemma \ref{lem:reg}]
Consider the equation $\partial_t f = \BB f$. Then from \eqref{eq:SSf} and \eqref{eq:BB0} we have
$$
\frac12\frac{d}{dt}\| f \|_{L^2(m_0)}^2 = - \int \bar a_{ij} \partial_i(m_0 f) \partial_j (m_0 f) + \int(\varphi_{m_0,2,1} - M\chi_R) m_0^2 f^2 .
$$
From Lemma~\ref{lem:bar-aij} there exists $\ell_0 >0$ such that  $\bar a_{ij} \xi_i \xi_j \geq \ell_0 |\xi|^2$. We obtain
\beqn\label{eq:reg}
\bal
\frac12\frac{d}{dt}\| f \|_{L^2(m_0)}^2 
&\leq - \ell_0 \int | \nabla (m_0 f)|^2   + \int(\varphi_{m_0,2,1} - M\chi_R) m_0^2 f^2 .
\eal 
\eeqn
The weight function $m_0$ satisfies \eqref{m2}, then Lemma~\ref{lem:hypo} holds, more precisely
\beqn\label{eq:hypo-m0}
\| \SS_\BB(t) f \|_{L^p(m_0)} \leq e^{at} \| f\|_{L^p(m_0)}, \qquad \forall\, t\geq 0.
\eeqn
Applying Nash's inequality in $3$-dimension: $ \| g \|_{L^2}^2 \leq c_1 \| \nabla g\|_{L^2}^{6/5} \| g\|_{L^1}^{4/5} $ with $g=m_0 f$ we obtain
$$
\bal
c_1^{-1} \| m_0 f \|_{L^2}^{10/3} \| m_0 f\|_{L^1}^{-4/3} 
&\leq \int |\nabla(m_0 f)|^2 .
\eal
$$
Putting together last inequality with \eqref{eq:reg}, it follows
\beqn\label{eq:reg2}
\bal
\frac12\frac{d}{dt}\|  f \|_{L^2(m_0)}^2
&\leq - C \, \| f \|_{L^2(m_0)}^{10/3} \| f\|_{L^1(m_0)}^{-4/3} + a \| f\|_{L^2(m_0)}^2.
\eal 
\eeqn
Let us denote $x(t) := \| f(t) \|_{L^2(m_0)}^2$ and $y(t) := \| f(t) \|_{L^1(m_0)}$ where $f(t) = \SS_\BB(t) f$. Then we have the following differential inequality $ \dot x(t) \le - C_1 x(t)^{5/3} y(t)^{-4/3} + 2a x(t)$.
From $\eqref{eq:hypo-m0}$ we have $y(t) \leq y_0$ and then 
$$ 
\dot x(t) \le - C_1 x(t)^{5/3} y_0^{-4/3} + 2a x(t).
$$
If $x_0 \leq C y_0$, by \eqref{eq:hypo-m0} we have $x(t) \leq C e^{at} y_0$. 
If $x_0$ is such that $x_0 > [C_1/4a] y_0$, then $x(t) \leq C (y_0^{-4/3} t)^{-3/2}$,
and we obtain 
$$
\| \SS_\BB(t) f \|_{L^2(m_0)} \leq C \,t^{-\frac34} e^{a t}\, \| f\|_{L^1(m_0)}.
$$
Using Riesz-Thorin interpolation theorem to $\SS_\BB(t)$ which acts from $L^2 \to L^2$ with estimate \eqref{eq:hypo-m0} and from $L^1 \to L^2$ with the estimate above, we obtain
\eqref{eq:regBB}.

Let us prove now \eqref{eq:regAB}. From Lemma~\ref{lem:A0} and equation \eqref{eq:chiR} we have the following estimates, for any $p\in[1,+\infty]$,
\beqn\label{eq:regAA}
\bal
\| \AA g \|_{L^2(\mu^{-1/2})} \lesssim \| g \|_{L^2(m_0)} ,
\qquad
\| \AA g \|_{L^p(m_0)} \lesssim \| g \|_{L^p(m)} .
\eal
\eeqn
Hence, by \eqref{eq:regAA} and \eqref{eq:regBB}, for $1\leq p \leq 2$, it follows
\beqn
\bal
\| \AA\SS_\BB(t) f \|_{L^2(\mu^{-1/2})} 
&\lesssim \| \SS_\BB(t) f \|_{L^2(m_0)} 
\lesssim t^{-\frac32\left(\frac1p - \frac12 \right)} \, e^{at}\, \| f \|_{L^p(m_0)}.
\eal
\eeqn
Computing the convolution of $\AA\SS_\BB(t)$ we have
$$
\bal
\| (\AA\SS_\BB)^{*2}(t) f \|_{L^2(\mu^{-1/2})} 
&\lesssim \int_0^t \|\AA\SS_\BB(t-s) \AA\SS_\BB(s) f \|_{L^2(\mu^{-1/2})} \, ds\\
&\lesssim \int_0^t \|\SS_\BB(t-s) \AA\SS_\BB(s) f \|_{L^2(m_0)} \, ds\\
&\lesssim \int_0^t (t-s)^{-\frac32\left(\frac1p - \frac12 \right)} e^{a(t-s)}\,\| \AA\SS_\BB(s) f \|_{L^p(m_0)} \, ds\\
&\lesssim \int_0^t (t-s)^{-\frac32\left(\frac1p - \frac12 \right)} e^{a(t-s)}\,\| \SS_\BB(s) f \|_{L^p(m)} \, ds\\
&\lesssim \int_0^t (t-s)^{-\frac32\left(\frac1p - \frac12 \right)} e^{a(t-s)}\, e^{as} \, \| f \|_{L^p(m)} \, ds\\
&\lesssim t^{\frac12\left(\frac72 - \frac3p \right)} e^{at}\,  \| f \|_{L^p(m)}\\
&\lesssim e^{a't} \, \| f \|_{L^p(m)},
\eal
$$
where we have used in order \eqref{eq:regAA}, \eqref{eq:regBB}, \eqref{eq:regAA}, Lemma~\ref{lem:hypo} and the fact that $(\frac72 - \frac3p ) >0$ for $1\leq p <2$. 
Hence, for all $t\geq 0$, we have $\| (\AA\SS_\BB)^{*2}(t)\|_{\BBB(L^p(m),L^2(\mu^{-1/2}))} \lesssim e^{a't}$, for any $a'>a > a_{m,p}$, where $a_{m,p}$ is defined in \eqref{amp1} and \eqref{amp2}.
\end{proof}


\subsection{Proof of Theorem \ref{thm:trou}}
With the results of Section~\ref{ssec:hypo}, Section~\ref{ssec:reg} and Theorem~\ref{thm:extension}, we are able to prove the semigroup decay for the linearised Landau operator.

\medskip
Let $E = L^2(\mu^{-1/2})$, in which space we already know the spectral gap \eqref{eq:lambda0bis}-\eqref{eq:lambda0}, which gives us assumption (1) of Theorem~\ref{thm:extension}. Let $\EE = L^p(m)$, for any $p\in[1,2]$ and $m$ satisfying \eqref{m1} or \eqref{m2}. We consider the decomposition $\LL = \AA + \BB$ as in \eqref{eq:AB}. For any $a>a_{m,p}$, the operator $\BB-a$ is hypo-dissipative in $\EE$ from Lemma~\ref{lem:hypo}, and this gives assumption \textit{(2i)} of Theorem~\ref{thm:extension}. Moreover, $\AA \in \BBB(\EE)$ and $A \in \BBB(E)$ from Lemma~\ref{lem:A0} and equation~\eqref{eq:chiR}, which gives assumption
\textit{(2ii)} of Theorem~\ref{thm:extension}. Hence we only need to prove assumption \textit{(2iii)} to conclude.

We split the proof into two different cases.

\medskip
\noindent
{\it Case $p=2$.}
In this case we have $E \subset \EE$.  Moreover,
$\AA\SS_\BB(t) \in \BBB(\EE,E)$ with exponential decay rate from Corollary \ref{cor:AB}, which proves assumption \textit{(2iii)} with $n=1$.

\medskip
\noindent
{\it Case $p\in [1,2)$.}
Here $E \subset \EE$ and from Lemma~\ref{lem:reg} we have $(\AA\SS_\BB)^{*2}(t) \in \BBB(\EE,E)$ with exponential decay rate, which gives assumption \textit{(2iii)} with $n=2$.


\section{Proof of the main result}\label{sec:nonlinear}

Recall the Landau operator \eqref{eq:oplandau}
$$
Q(g,h) = (a_{ij}*g)\partial_{ij}h - (c*g)h.
$$
We shall prove some estimates for the nonlinear operator $Q$ before proving the Theorem~\ref{thm:rate}.

\begin{prop}\label{prop:EstimationPotentielDur}
Let $\gamma\in[0,1]$ and $p\in [1,+\infty]$. Then
$$
\| Q(g,h) \|_{L^p(m)}
\lesssim  
\|  g\|_{L^1(\la v\ra^{\gamma+2})}  \| \partial_{ij} h\|_{L^p(m\la v\ra^{\gamma+2})}
+
\| g\|_{L^{1}(\la v\ra^{\gamma})}
\| h \|_{L^p(m \la v \ra^{\gamma})}
$$

\end{prop}

\begin{proof}[Proof of Proposition~\ref{prop:EstimationPotentielDur}]
We write
$$
\|  Q(g,h) \|_{L^p(m)}
\leq 
\| (a_{ij}*g)\partial_{ij} h \|_{L^p(m)} 
+ \| (c*g) h \|_{L^p(m)}.
$$
Thanks to Lemma~\ref{lem:a*g}
$$
\bal
\| (a_{ij}*g) \partial_{ij} h \|_{L^p(m)}
&\lesssim 
\|  g\|_{L^1(\la v\ra^{\gamma+2})}  \| \partial_{ij} h\|_{L^p(m\la v\ra^{\gamma+2})}
\eal
$$
Moreover, by Lemma~\ref{lem:a*g} one obtains, since $c = \partial_{ij} a_{ij}$ and $|(c * g)(v)| \leq C \la v \ra^{\gamma} \|  g\|_{L^{1}(\la v\ra^{\gamma})}$,  
$$
\| (c * g)  h \|_{L^p(m)}
\lesssim  \| g\|_{L^{1}(\la v\ra^{\gamma})}
\| h \|_{L^p(m \la v \ra^{\gamma})},
$$
and the proof is complete.
\end{proof}


The proof of Theorem~\ref{thm:rate} relies on known results by Desvillettes and Villani \cite{DesVi1,DesVi2} concerning the polynomial decay rate to equilibrium, together with the semigroup decay estimates from Theorem~\ref{thm:trou} and some estimates on the nonlinear operator from Proposition~\ref{prop:EstimationPotentielDur}. We follow the strategy developed 
in \cite{Mouhot2}.

\medskip

Let us first summarise the results on the Cauchy theory for the Landau equation with hard potentials from \cite[Theorems 3, 6 and 7]{DesVi1} and \cite[Theorem 8]{DesVi2}, with a improvement of \cite{safadi} concerning the smoothness effect.

\begin{thm}\label{thm:DV}
Consider $\gamma \in (0,1]$.

\begin{enumerate}[(1)]

\item  Let $f_0 \in L^1(\la v\ra^{2+\delta})$ for some $\delta>0$ and consider a weak solution $f$ to \eqref{eq:landau}, then:

(a) for all $t_0 > 0$, all integer $k>0$ and all $\theta>0$, there exists $C_{t_0}>0$ such that 
$$
\sup_{t\geq t_0} \| f(t,\cdot) \|_{H^k (\la v\ra^{\theta})} \leq C_{t_0}.
$$

(b) for all $t_0 > 0$, $f \in \CC^{\infty} ([t_0, + \infty); \SS(\R^3_v))$.

\item Let $f$ be any weak solution of \eqref{eq:landau} with initial datum $f_0\in L^1(\la v\ra^2)$ satisfying the decay of energy, then for all $t_0 >0$ and all $\theta>0$, there is a constant $C_{t_0}>0$ such that
$$
\sup_{t\geq t_0} \| f(t,\cdot) \|_{L^1 (\la v\ra^{\theta})} \leq C_{t_0}.
$$

\item  If $f$ is a smooth solution of \eqref{eq:landau} (in the sense of (1) above), then for all $t\geq 0$ there is $C>0$ such that
$$
H(f_t | \mu ) := \int_{\R^3} f_t \log \frac{f_t}{\mu}\, dv \leq C (1+t)^{-2/\gamma} 
$$

\end{enumerate}

\end{thm}

\begin{cor}\label{cor:poly-rate}
For all $t_0 >0$ and all $\ell >0$, there exists $C_{t_0}>0$ such that
$$
\forall\, t\geq t_0, \qquad
\| f_t - \mu \|_{L^1(\la v \ra^\ell)} \leq C_{t_0} (1+t)^{-\frac{1}{2\gamma}}.
$$
\end{cor}

\begin{proof}[Proof of Corollary~\ref{cor:poly-rate}]
Let us fixe some $t_0 >0$. First of all, from Theorem \ref{thm:DV} and the Csisz\'ar-Kullback-Pinsker inequality  (see e.g.\ \cite[Remark 22.12]{VillaniOTO&N})
$$
\| f - \mu \|_{L^1(\R^3)} \leq C \sqrt{ H(F | \mu) },  
$$
we obtain
\beqn\label{eq:convL1}
\forall\, t\geq 0, \qquad
\| f_t - \mu \|_{L^1 (\R^3)} \leq C (1+t)^{-1/\gamma}.
\eeqn
Then, using the bounds of Theorem~\ref{thm:DV} and H\"older's inequality we obtain
$$
\forall\, t\geq t_0, \qquad
\| f_t - \mu \|_{L^1(\la v\ra^{\ell})} 
\leq    \| f_t - \mu \|_{L^1(\la v\ra^{2\ell})}^{1/2} \| f_t - \mu \|_{L^1(\R^3)}^{1/2}
\leq C_{t_0} (1+t)^{-\frac{1}{2\gamma}}.
$$
\end{proof}

Let $f = \mu + h$, then $h=h(t,v)$ satisfies the equation
\beqn\label{eq:cauchylineaire}
\left\{
\bal
\partial_t h &= \LL h + Q(h,h)\\
h_{|t=0} &= h_0 = f_0 - \mu  .
\eal
\right.
\eeqn
Since $f_0 = \mu + h_0$ has same mass, momentum and energy than $\mu$, we have $\Pi h_0 = 0$ and for all $t\geq 0$, thanks to the conservation of these quantities, we also have $\Pi h_t = \Pi Q(h_t,h_t)= 0$.

Before giving the proof of Theorem~\ref{thm:rate}, we state and prove the following lemma which will be important for the sequel.

\begin{lem}\label{lem:exp-rate}
Consider $m=\la v\ra^k$ satisfying \eqref{m1}. There exists $\epsilon >0$ such that, if the solution $h$ of \eqref{eq:cauchylineaire} satisfies
$$
\| h_0 \|_{L^1(\la v \ra^k)} \le \epsilon \quad\text{and}\quad
\| h_t \|_{L^1(\la v\ra^{\ell})} \leq \epsilon, \quad \forall\, t\geq 0,
$$
with $\ell := 2\gamma+8 + k$, and if 
$$
\forall\, t\geq 0, \qquad
\| h_t \|_{H^4(\la v\ra^{\ell})} \leq C,
$$
then there is $C'>0$ such that
$$
\forall\, t\geq 0, \qquad
\| h_t \|_{L^1(\la v\ra^{k})} \leq C' e^{-\lambda_0 t}\, \| h_0 \|_{L^1(\la v\ra^{k})},
$$
where $\lambda_0>0$ is the spectral gap in \eqref{eq:lambda0bis}-\eqref{eq:lambda0}.
\end{lem}

\begin{proof}[Proof of Lemma~\ref{lem:exp-rate}]
By Duhamel's formula for the solution of \eqref{eq:cauchylineaire}, we write, 
$$
h_t = S_{\LL}(t) h_{0} + \int_{t_0}^t S_{\LL}(t-s) Q(h_s,h_s)\, ds.
$$
Using Theorem~\ref{thm:trou} (observe that we can take $\lambda = \lambda_0$ in that theorem since $\gamma \in (0,1]$, see Remark~\ref{rem:trou}) and Proposition~\ref{prop:EstimationPotentielDur}, one deduces
\beqn\label{eq:estL1}
\bal
\| h_t\|_{L^1(\la v\ra^k)} 
&\leq  \| S_{\LL}(t) h_0 \|_{L^1(\la v\ra^k)} + \int_{0}^t \|S_{\LL}(t-s) Q(h_s,h_s)\|_{L^1(\la v\ra^k)}\, ds \\
&\leq C e^{-\lambda_0 t}\| h_0 \|_{L^1(\la v\ra^k)} + C \int_{0}^t e^{-\lambda_0(t-s)}\| Q(h_s,h_s)\|_{L^1(\la v\ra^k)}\, ds \\
&\leq C e^{-\lambda_0 t}\| h_0 \|_{L^1(\la v\ra^k)}  
+C\int_{0}^t e^{-\lambda_0(t-s)} 
 \Big( \| h_s \|_{L^1(\la v \ra^{\gamma})}\| h_s \|_{L^1(\la v \ra^{\gamma+k})} \\
&\qquad\qquad\qquad\qquad \qquad\qquad
+ \| h_s \|_{L^1(\la v \ra^{\gamma+2})}   \| \nabla^2 h_s \|_{L^1(\la v \ra^{\gamma+2+k})} \Big)\, ds .
\eal
\eeqn
We recall the following interpolation inequality from \cite[Lemma B.1]{MiMo-cmp}
$$
\| u \|_{W^{q,1}(\la v \ra^{\alpha})} \le C \| u \|_{W^{q_1,1}(\la v \ra^{\alpha_1})}^{1-\theta} \, 
\| u \|_{W^{q_2,1}(\la v \ra^{\alpha_2})}^{\theta}
$$
with $\theta \in (0,1)$, $\alpha \ge \alpha_1$ and $q \ge q_1$, $q = (1-\theta)q_1 + \theta q_2$ and $ \alpha = (1-\theta)\alpha_1 + \theta \alpha_2 $ with $q,q_1,q_2,\alpha,\alpha_1, \alpha_2 \in \Z$. From this we get
$$
\bal
\| \nabla^2  h \|_{L^1(\la v \ra^{\gamma+2+k})} 
\lesssim \| h \|_{L^1(\la v \ra^k)}^{1/2} \, \| h \|_{W^{4,1} (\la v \ra^{2\gamma + 4 + k})}^{1/2}
\lesssim \| h \|_{L^1(\la v \ra^k)}^{1/2} \, \| h \|_{H^{4} (\la v \ra^{2\gamma + 6 + k})}^{1/2},
\eal
$$
where we used H\"older's inequality in last step.
Gathering last inequality with \eqref{eq:estL1} and using H\"older's inequality again to write 
$$
\| h \|_{L^1(\la v \ra^{\gamma})} \, \| h \|_{L^1(\la v \ra^{\gamma+k})} 
\le \| h \|_{L^1(\la v \ra^{2\gamma + k})}^{1/2} \, \| h \|_{L^1(\la v \ra^{k})}^{3/2},
$$
it follows that
$$
\bal
\| h_t \|_{L^1(\la v\ra^k)} 
&\leq C e^{-\lambda_0 t}\| h_0 \|_{L^1(\la v\ra^k)} + C \int_{0}^t e^{-\lambda_0(t-s)} 
  \| h_s \|_{L^1(\la v \ra^{2\gamma + k})}^{1/2} \, \| h_s \|_{L^1(\la v \ra^{k})}^{3/2} \, ds \\
&\quad   
+C  \int_0^t e^{-\lambda_0(t-s)} \| h_s \|_{H^4(\la v \ra^{2\gamma + 6 + k})}^{1/2} \, \| h_s \|_{L^1(\la v \ra^{k})}^{3/2} \, ds.
\eal
$$
Denoting $x(t) := \| h_t \|_{L^1(\la v\ra^k)}$ and using the assumptions of the lemma, we obtain the following inequality
$$
x(t) \leq C e^{-\lambda_0 t} x(0) 
+ C \epsilon^{1/4} \int_0^t  e^{-\lambda_0(t-s)} x(s)^{1+1/4}\, ds.
$$
Arguing as in \cite[Lemma 4.5]{Mouhot2}, if $x(0)$ and $\epsilon$ are small enough we obtain, for all $t\geq 0$, $x(t) \leq C' e^{-\lambda_0 t} x(0)$, i.e.
$$
\| h_t \|_{L^1(\la v \ra^{k})} \leq C' e^{-\lambda_0 t} \| h_0 \|_{L^1(\la v \ra^{k})}.
$$
\end{proof}

\begin{proof}[Proof of Theorem \ref{thm:rate}]

We can now complete the proof of Theorem \ref{thm:rate}. From Corollary \ref{cor:poly-rate}, we pick $t_0 >0$ such that
$$
\forall\, t\geq t_0 ,\qquad
\| f_t - \mu\|_{L^1(\la v\ra^\ell)} = \| h_t \|_{L^1(\la v\ra^\ell)} \leq \epsilon,
$$
where $\epsilon$ is chosen in Lemma~\ref{lem:exp-rate}. From Theorem~\ref{thm:DV} we have that, for all $t\geq t_0$, 
$$
\| h_t \|_{H^4(\la v \ra^\ell)} \leq  \| f_t \|_{H^4(\la v \ra^\ell)} + \| \mu \|_{H^4(\la v \ra^\ell)} \leq C.
$$
We can then apply Lemma~\ref{lem:exp-rate} to $h_t$ starting from $t_0$, then
$$
\forall\, t\geq t_0 ,\qquad
\| f_t - \mu\|_{L^1(\la v\ra^k)} = \| h_t \|_{L^1(\la v\ra^k)} \leq C' e^{-\lambda_0 t} \| h_{t_0} \|_{L^1(\la v \ra^k)}
\leq C'' e^{-\lambda_0 t}.
$$
This last estimate together with \eqref{eq:convL1} for $ t \in [0 , t_0]$ completes the proof. 
\end{proof}

\bibliographystyle{acm}
\bibliography{bib-spectral}

\end{document}